\documentclass[12pt]{amsart}
\usepackage{graphicx}
\usepackage[T1]{fontenc}
\usepackage{color}

\newtheorem{thm}{Theorem}[section]

\newtheorem{prop}[thm]{Proposition}

\theoremstyle{definition}

\theoremstyle{remark}
\newtheorem{remark}[thm]{Remark}

\numberwithin{equation}{section}

\newcommand{\dd}{\mathrm{d}}
\newcommand{\ee}{\mathrm{e}}
\newcommand{\rr}{{\mathbb R}}

\newcommand{\eqd}{\stackrel{\rm d}{=}}

\newcommand{\Exp}{\mathbb E}
\newcommand{\Cov}{\operatorname{Cov}}
\newcommand{\arsinh}{\operatorname{Arsinh}}

\newcommand{\cD}{{\mathcal D}}
\newcommand{\distre}{\stackrel{\cD}{=}}

\allowdisplaybreaks

\begin{document}

\sloppy

\title[A link between identities for the Wiener process and its bridge]{A link between Bougerol's identity
 and a formula due to Donati-Martin, Matsumoto and Yor}


\author{M\'aty\'as Barczy}
\address{M\'aty\'as Barczy, Faculty of Informatics, University of Debrecen, Pf.\ 12, H-4010 Debrecen, Hungary}
\email{barczy.matyas\@@{}inf.unideb.hu}

\author{Peter Kern}
\address{Peter Kern, Mathematical Institute, Heinrich-Heine-University D\"usseldorf, Universit\"atsstr.\ 1, D-40225 D\"usseldorf, Germany}
\email{kern\@@{}math.uni-duesseldorf.de}

\date{\today}

\begin{abstract}
We point out an easy link between two striking identities on exponential functionals of the Wiener process and the Wiener bridge originated by Bougerol, and Donati-Martin, Matsumoto and Yor, respectively.
The link is established using a continuous one-parameter family of Gaussian processes known as $\alpha$-Wiener bridges  or scaled Wiener bridges, which in case $\alpha=0$ coincides with a Wiener process and for $\alpha=1$ is a version of the Wiener bridge.
\end{abstract}



\maketitle

\baselineskip=18pt

\section{Introduction}%

Our starting point is Bougerol's identity in \cite{Bou} which states that
\begin{equation}\label{Bougerol}
\sinh(B_t)\eqd W_{A_t}\quad\text{ for every fixed }  t\geq0,
\end{equation}
where $(B_t)_{t\geq0}$ and $(W_t)_{t\geq0}$ are independent standard Wiener processes, $\eqd$ denotes equality in distribution, and
$$ A_t=\int_0^t\exp(2\,B_s)\,\dd s\quad\text{ for }t\geq0.$$
In fact there is also a generalization of Bougerol's identity with equality in law for stochastic processes due to Alili, Dufresne and Yor \cite[Proposition 2]{ADY}; cf.\ also \cite[formula (69)]{Vak} or \cite[page 200]{Yor2}.
Recently, there has been a renewed interest in generalizations of Bougerol's identity \eqref{Bougerol}.
Bertoin et al. \cite{BerDufYor} presented a two-dimensional extension of \eqref{Bougerol} that involves
 some exponentional functional and the local time at $0$ of a standard Wiener process.
For another two-dimensional extension of \eqref{Bougerol}, and even a three-dimensional one we refer to Vakeroudis \cite[Sections 4.2 and 4.3]{Vak}.

We are only interested in the following particular case of the identity \eqref{Bougerol} presented in \cite{Vak,Yor1}.
Bougerol's identity \eqref{Bougerol} is equivalent to the equality of the corresponding continuous Lebesgue densities, which yields
$$\frac1{\sqrt{(1+x^2)t}}\exp\left(-\frac{\arsinh^2(x)}{2t}\right)=\Exp\left[\frac1{\sqrt{A_t}}\exp\left(-\frac{x^2}{2A_t}\right)\right]$$
for all $t>0$ and $x\in\rr$, see, e.g., \cite[formula (1.e)]{Yor1}.
Especially, for $x=0$, by the $1/2$-self-similarity of a standard Wiener process and a change of variables $r=(4/\beta^2)s$ for some $\beta>0$ we get
\begin{equation}\label{Bougerol_gen}\begin{split}
t^{-1/2} & =\Exp\left[\left(\int_0^t\exp(2B_s)\,\dd s\right)^{-1/2}\right]=\Exp\left[\left(\int_0^{t}\exp(\beta B_{(4/\beta^2)s})\,\dd s\right)^{-1/2}\right]\\
& =\frac2\beta\cdot\Exp\left[\left(\int_0^{(4/\beta^2)t}\exp(\beta B_{r})\,\dd r\right)^{-1/2}\right].
\end{split}\end{equation}
Hence, setting
$t=\beta^2/4$ we get for every $\beta>0$
\begin{equation*}
\Exp\left[\left(\int_0^1\exp(\beta B_{s})\,\dd s\right)^{-1/2}\right]=1.
\end{equation*}
This formula is a consequence of Bougerol's identity \eqref{Bougerol}
 which obviously holds for $\beta=0$ and also remains true for $\beta<0$, since $(-B_t)_{t\geq0}$ is a Wiener process, i.e.,
\begin{equation}\label{Bougerol_special}
\Exp\left[\left(\int_0^1\exp(\beta B_{s})\,\dd s\right)^{-1/2}\right]=1\quad\text{ for every }\beta\in\rr.
\end{equation}

A similar identity due to Donati-Martin, Matsumoto and Yor \cite{DMMY1,DMMY2} holds when replacing the Wiener process $(B_t)_{t\geq0}$ by a Wiener bridge $(B_t^\circ=B_t-t\,B_1)_{t\in[0,1]}$, a zero mean Gaussian process with covariance function $\Cov(B_s^\circ,B_t^\circ)=s(1-t)$ for $0\leq s\leq t\leq 1$. Namely, this identity states that
\begin{equation}\label{DMMY}
\Exp\left[\left(\int_0^1\exp(\beta B_{s}^\circ)\,\dd s\right)^{-1}\right]=1\quad\text{ for every }\beta\in\rr.
\end{equation}
Hobson \cite{Hob} provides a simple proof of \eqref{DMMY} using a relationship between a Wiener bridge and a Wiener excursion obtained by Biane \cite{Bia}. A further elementary proof of \eqref{DMMY} is given in \cite[Proposition 2.1]{DMMY1}.

Donati-Martin et al. \cite{DMMY1} already pointed out how to obtain a link between the two identities \eqref{Bougerol_special} and \eqref{DMMY} in the sense that the identity \eqref{Bougerol_special} follows from the identity \eqref{DMMY} as a consequence of a formula combining exponential functionals of the Wiener process and the Wiener bridge, for details we refer to \cite[Proposition 3.2]{DMMY1}.

Our aim is to give a different link between the two identities \eqref{Bougerol_special} and \eqref{DMMY} using so-called $\alpha$-Wiener bridges (also known as scaled Wiener bridges). These processes build a one-parameter family of Gaussian processes for parameter $\alpha\in\rr$. They have been first considered by Brennan and Schwartz \cite{BreSch} and later have been investigated by Mansuy \cite{Man} and Barczy and Pap \cite{BarPap}. For our purposes an $\alpha$-Wiener bridge $(X^{(\alpha)}_t)_{t\in[0,1)}$ can be  defined as a (weak) solution of the stochastic differential equation (SDE)
\begin{equation}\label{aWSDE}
\dd X^{(\alpha)}_t=-\frac{\alpha}{1-t}X^{(\alpha)}_t\,\dd t+\dd B_t,\qquad t\in[0,1),
\end{equation}
with initial condition $X^{(\alpha)}_0=0$. Barczy and Pap \cite{BarPap} have shown that $(X^{(\alpha)}_t)_{t\in[0,1)}$ is a bridge in the sense that $X^{(\alpha)}_t\to0=:X_1^{(\alpha)}$ as $t\uparrow1$ almost surely if and only if $\alpha>0$. Moreover, for $\alpha\geq0$ it is shown in \cite{BarPap} that $(X^{(\alpha)}_t)_{t\in[0,1]}$ is a zero mean Gaussian process with covariance function
\begin{equation}\label{aWCov}
\Cov(X^{(\alpha)}_s,X^{(\alpha)}_t)=\begin{cases}
\frac{(1-s)^\alpha(1-t)^\alpha}{1-2\alpha}\left(1-(1-s)^{1-2\alpha}\right) & \text{ if }\alpha\not=\tfrac12\\
\sqrt{(1-s)(1-t)}\log\left(\frac1{1-s}\right) & \text{ if }\alpha=\tfrac12
\end{cases}
\end{equation}
for $0\leq s\leq t\leq 1$. Note that for fixed $0\leq s\leq t\leq 1$, \eqref{aWCov} is continuous in $\alpha\geq0$, which for $\alpha\to\frac12$ can be easily seen by l'Hospital's rule. The unique strong solution of the SDE \eqref{aWSDE} with initial condition $X^{(\alpha)}_0=0$ is given by
\begin{equation}\label{intrep}
X^{(\alpha)}_t=\int_0^t\left(\frac{1-t}{1-s}\right)^\alpha\,\dd B_s\quad\text{ for }t\in[0,1),
\end{equation}
and shows that $(X^{(0)}_t)_{t\in[0,1]}=(B_t)_{t\in[0,1]}$ and $(X^{(1)}_t)_{t\in[0,1]}\eqd(B^\circ_t)_{t\in[0,1]}$. The latter is due to the fact that both sides of the equation are zero mean Gaussian processes with the same covariance function. Hence, variation of the parameter $\alpha\in[0,1]$ continuously connects the Wiener process for $\alpha=0$ with the Wiener bridge for $\alpha=1$ in the sense that for
 $\alpha,\,\alpha_0\geq0$, the finite dimensional distributions of $(X^{(\alpha)})_{t\in[0,1]}$ converge weakly to those of $(X^{(\alpha_0)})_{t\in[0,1]}$ as $\alpha\to\alpha_0$. This follows directly from the continuity in $\alpha$ of the covariance function \eqref{aWCov} and is the key observation for our link between the identities \eqref{Bougerol_special} and \eqref{DMMY}.

The paper is organized as follows.
We will first show that certain space-time rescalings of an $\alpha$-Wiener bridge either coincide in law with a usual Wiener bridge for $\alpha>\frac12$ or with the Wiener process for $0\leq\alpha<\frac12$, see Proposition \ref{spatim}. Then an application of these space-time rescalings to the dentity \eqref{DMMY} and \eqref{Bougerol_special},  respectively, yields two new identities for certain transformations of exponential functionals of $\alpha$-Wiener bridges which coincide when $\alpha=\frac12$, see Theorem \ref{link}. We further show that a $\frac{1}{2}$-Wiener bridge can be scaled to both, a Wiener bridge and a standard Wiener process, see Proposition \ref{sts12}. As a consequence,
we present another two identities for certain transformations of exponential functionals of $\frac{1}{2}$-Wiener bridges in Theorem \ref{link12}.

\section{Link between the identities}

In the sequel, $\distre$ denotes equality in law for stochastic processes on the space of continuous functions $C([0,1])$ or $C([0,\infty))$, respectively.
\begin{prop}\label{spatim}
(a) For $\alpha>\frac12$ we have
$$\left(\sqrt{2\alpha-1}\,t^{\frac{\alpha-1}{2\alpha-1}}X^{(\alpha)}_{1-t^{1/(2\alpha-1)}}\right)_{t\in[0,1]}{\distre}(X^{(1)}_t)_{t\in[0,1]}.$$
(b) For $0\leq\alpha<\frac12$ we have
$$\left(\sqrt{1-2\alpha}\,(1-t)^{-\frac{\alpha}{1-2\alpha}}
   X^{(\alpha)}_{1-(1-t)^{1/(1-2\alpha)}}\right)_{t\in[0,1]}{\distre}(X^{(0)}_t)_{t\in[0,1]}.$$
\end{prop}
\begin{proof}
We will first prove that the processes under consideration are zero mean Gaussian processes having almost surely continuous trajectories, which is not obvious for the left-hand sides as $t\downarrow0$ for $\alpha\in(\frac12,1)$ in (a), and as $t\uparrow1$ in (b), respectively. Once we know this, it remains to show the equality of covariance functions. \\
(a) Let $(M_t)_{t\in[0,1)}$ be the continuous martingale part of the process $X^{(\alpha)}$ given by \eqref{intrep}
$$M_t:=\frac{X_t^{(\alpha)}}{(1-t)^\alpha}=\int_0^t\frac1{(1-s)^\alpha}\,\dd B_s \quad \text{ for } t\in[0,1)$$
with quadratic variation $\langle M\rangle_t=(1-(1-t)^{1-2\alpha})/(1-2\alpha)\to\infty$ as $t\uparrow1$ for $\alpha>\frac12$ as obtained in \cite[formula (3.1)]{BarPap}. Then, similarly to the proof of \cite[Lemma 3.1]{BarPap}, for the increasing function $[1,\infty)\ni x\mapsto f(x)=x^{3/4}$ with $\int_1^\infty(f(x))^{-2}\,\dd x<\infty$, an application of \cite[Theoreme 1]{Lep} or Exercise 1.16 in Chapter V of \cite{RY} gives $M_t/f(\langle M\rangle_t)\to0$ a.s.\ as $t\uparrow1$. Letting $t=1-s^{1/(2\alpha-1)}\uparrow1$ as $s\downarrow0$ this shows
$$\frac{s^{\frac{-\alpha}{2\alpha-1}}X_{1-s^{1/(2\alpha-1)}}^{(\alpha)}}{\big((1-s^{-1})/(1-2\alpha)\big)^{3/4}}\to0\quad\text{ a.s.\ as  }s\downarrow0.$$
To obtain $s^{\frac{\alpha-1}{2\alpha-1}}X^{(\alpha)}_{1-s^{1/(2\alpha-1)}}\to0$ a.s. as $s\downarrow0$ it suffices to see that for $s\downarrow0$ we have
$$s^{\frac{\alpha-1}{2\alpha-1}}s^{\frac{\alpha}{2\alpha-1}}({s^{-1} - 1})^{3/4}=s({s^{-1} - 1})^{3/4}=s^{1/4}({1 - s})^{3/4}\to0.$$
Hence the centered Gaussian processes under consideration almost surely have continuous sample paths on $[0,1]$ starting in the origin. Thus it remains to show the equality of their covariance functions for $0<s\leq t\leq1$. Using \eqref{aWCov} and the fact that the function $(0,1]\ni t\mapsto 1-t^{1/(2\alpha-1)}$ is decreasing, we get for $0<s\leq t\leq1$
\begin{align*}
& \Cov\left(X^{(\alpha)}_{1-s^{1/(2\alpha-1)}},X^{(\alpha)}_{1-t^{1/(2\alpha-1)}}\right)=\frac{s^{\frac{\alpha}{2\alpha-1}}t^{\frac{\alpha}{2\alpha-1}}}{1-2\alpha}\,(1-t^{-1})\\
& \quad=\frac{s^{\frac{\alpha}{2\alpha-1}}t^{\frac{\alpha}{2\alpha-1}-1}}{2\alpha-1}\,(1-t)=\frac{s^{\frac{1-\alpha}{2\alpha-1}}t^{\frac{1-\alpha}{2\alpha-1}}}{2\alpha-1}\,s(1-t)
\end{align*}
from which the assertion easily follows.\\
(b) In case $\alpha=0$ the identity is trivially fulfilled. For $0<\alpha<\frac12$ it is shown in the proof of \cite[Lemma 3.1]{BarPap} that $\lim_{t\uparrow1}(1-t)^{-\alpha}X_t^{(\alpha)}$ exists in $\rr$ almost surely and has a normal distribution as a limit of normally distributed random variables. Letting $t=1-(1-s)^{1/(1-2\alpha)}\uparrow1$ as $s\uparrow1$ we have
$$\lim_{s\uparrow1}(1-s)^{-\frac{\alpha}{1-2\alpha}}X^{(\alpha)}_{1-(1-s)^{1/(1-2\alpha)}}\quad\text{exists a.s.,}$$
which shows that the centered Gaussian processes under consideration almost surely have continuous sample paths on $[0,1]$ starting in the origin. Thus it remains to show the equality of their covariance functions. Using \eqref{aWCov} and the fact that the function $[0,1]\ni t\mapsto 1-(1-t)^{1/(1-2\alpha)}$ is increasing, we get for $0\leq s\leq t\leq1$
\begin{align*}
& \Cov\left(X^{(\alpha)}_{1-(1-s)^{1/(1-2\alpha)}},X^{(\alpha)}_{1-(1-t)^{1/(1-2\alpha)}}\right)=\frac{(1-s)^{\frac{\alpha}{1-2\alpha}}(1-t)^{\frac{\alpha}{1-2\alpha}}}{1-2\alpha}\,s
\end{align*}
from which again the assertion easily follows.
\end{proof}
\begin{thm}\label{link}
(a) For $\alpha>\frac12$ and any $\beta\in\rr$ we have
$$\Exp\left[\left(\int_0^1\exp\left(\frac{\beta}{(1-s)^{1-\alpha}}\,X^{(\alpha)}_s\right)\,\frac{\dd s}{(1-s)^{2(1-\alpha)}}\right)^{-1}\right]=2\alpha-1.$$
(b) For $0\leq\alpha<\frac12$ and any $\beta\in\rr$ we have
$$\Exp\left[\left(\int_0^1\exp\left(\frac{\beta}{(1-s)^{\alpha}}\,X^{(\alpha)}_s\right)\,\frac{\dd s}{(1-s)^{2\alpha}}\right)^{-1/2}\right]=\sqrt{1-2\alpha}.$$
(c) For $\alpha=\frac12$ and any $\beta\in\rr$ both identities in (a) and (b) hold.
\end{thm}
\begin{remark}
For the $\frac12$-Wiener bridge the two identities in (a) and (b) of Theorem \ref{link} are valid by part (c) and are in fact equivalent, since both identities show that for any $\beta\in\rr$ the non-negative random variable
$$Y(\beta):=\left(\int_0^1\exp\left(\frac{\beta}{\sqrt{1-s}}\,X^{(1/2)}_s\right)\,\frac{\dd s}{1-s}\right)^{-1/2}=0\quad\text{ almost surely.}$$
Hence the version of the Bougerol identity in (b) represents the mean $\Exp[Y(\beta)]=0$, whereas the formula (a),
 as a version of the identity due to Donati-Martin, Matsumoto and Yor, represents the second moment $\Exp[(Y(\beta))^2]=0$.
\end{remark}
\begin{proof}[Proof of Theorem \ref{link}]
(a) An application of Proposition \ref{spatim} (a) to \eqref{DMMY} together with a change of variables $s=1-t^{\frac1{2\alpha-1}}$ yields for any $\beta\in\rr$
\begin{align*}
1 & =\Exp\left[\left(\int_0^1\exp(\beta X^{(1)}_t)\,\dd t\right)^{-1}\right]\\
& =\Exp\left[\left(\int_0^1\exp\left(\beta \sqrt{2\alpha-1}\,t^{\frac{\alpha-1}{2\alpha-1}}X^{(\alpha)}_{1-t^{1/(2\alpha-1)}}\right)\,\dd t\right)^{-1}\right]\\
& =\Exp\left[\left(\int_0^1\exp\left(\frac{\tilde\beta}{(1-s)^{1-\alpha}}\,X^{(\alpha)}_s\right)\cdot(2\alpha-1)\,\frac{\dd s}{(1-s)^{2(1-\alpha)}}\right)^{-1}\right],
\end{align*}
where $\tilde\beta=\beta\sqrt{2\alpha-1}\in\rr$ is arbitrary.\\
(b) For $\alpha=0$ the identity is a restatement of \eqref{Bougerol_special}. For $0<\alpha<\frac12$ an application of Proposition \ref{spatim} (b) to \eqref{Bougerol_special} together with a change of variables $s=1-(1-t)^{1/(1-2\alpha)}$ yields for any $\beta\in\rr$
\begin{align*}
1 & =\Exp\left[\left(\int_0^1\exp(\beta X^{(0)}_t)\,\dd t\right)^{-1/2}\right]\\
& =\Exp\left[\left(\int_0^1\exp\left(\beta \sqrt{1-2\alpha}\,(1-t)^{-\frac{\alpha}{1-2\alpha}}X^{(\alpha)}_{1-(1-t)^{1/(1-2\alpha)}}\right)\,\dd t\right)^{-1/2}\right]\\
& =\Exp\left[\left(\int_0^1\exp\left(\frac{\tilde\beta}{(1-s)^{\alpha}}\,X^{(\alpha)}_s\right)\cdot(1-2\alpha)\,\frac{\dd s}{(1-s)^{2\alpha}}\right)^{-1/2}\right],
\end{align*}
where $\tilde\beta=\beta\sqrt{1-2\alpha}\in\rr$ is arbitrary.\\
(c) For $\alpha=\frac12$ the process $(M_t)_{t\in[0,1)}$ with $M_t=(1-t)^{-1/2}X_t^{(1/2)}=\int_0^t(1-s)^{-1/2}\dd B_s$ is a centered continuous martingale with quadratic variation $\langle M\rangle_t=-\log(1-t)\to\infty$ as $t\uparrow1$; see formulas (3.1) and (3.2) in \cite{BarPap}. Hence by the Dambis, Dubins-Schwarz theorem there exists a Wiener process $(\tilde B_t)_{t\geq0}$ such that $(M_t)_{t\in[0,1)}=(\tilde B_{\langle M\rangle_t})_{t\in[0,1)}$ almost surely; see Theorem 1.6 in Chapter V of \cite{RY}.
It follows by a change of variables $t=\langle M\rangle_s=-\log(1-s)$ and monotone convergence that for $\beta\not=0$
\begin{align*}
& \Exp\left[\left(\int_0^1\exp\left(\frac{\beta}{\sqrt{1-s}}\,X^{(1/2)}_s\right)\,\frac{\dd s}{1-s}\right)^{-1/2}\right]\\
& \quad=\Exp\left[\left(\int_0^1\exp\big(\beta\tilde B_{-\log(1-s)}\big)\,\frac{\dd s}{1-s}\right)^{-1/2}\right]=\Exp\left[\left(\int_0^\infty\exp\big(\beta\tilde B_{t}\big)\,\dd t\right)^{-1/2}\right]\\
& \quad=\lim_{T\to\infty}\Exp\left[\left(\int_0^T\exp\big(\beta\tilde B_{t}\big)\,\dd t\right)^{-1/2}\right]=\lim_{T\to\infty}T^{-1/2}=0,
\end{align*}
where the last but one equality follows by setting $t=\beta^2T/4$ in \eqref{Bougerol_gen}.
Since in case $\beta=0$ the expectation is obviously vanishing, this shows that the identity in (b) is fulfilled for $\alpha=\frac12$. In particular it shows that a non-negative random variable has zero expectation and thus is equal to zero almost surely. Hence also its second moment vanishes, which proves the identity in (a) for $\alpha=\frac12$.
\end{proof}
In case $\alpha=\frac12$ it is possible to link the $\frac12$-Wiener bridge $(X^{(1/2)}_t)_{t\in[0,1]}$ to both identities \eqref{DMMY}
 and \eqref{Bougerol_special} with non-vanishing expectation by either introducing an additional $\log$-term in the integrand or by integrating over a smaller domain as follows. We first present the corresponding space-time scalings, which might be of independent interest.
\begin{prop}\label{sts12}
We have
\begin{align}
& \left(t\,\sqrt{\exp(t^{-1}-1)}X^{(1/2)}_{1-\exp(1-t^{-1})}\right)_{t\in[0,1]}{\distre}(X^{(1)}_t)_{t\in[0,1]}.\label{spatim12DMMY}\\
& \left(\ee^{t/2}X^{(1/2)}_{1-{\exp(-t)}}\right)_{t\geq0}{\distre}(X^{(0)}_t)_{t\geq0}.\label{spatim12Bou}
\end{align}
\end{prop}
\begin{proof}
We first show that as $t\downarrow0$ we have
\begin{equation}\label{cont12}
t\,\sqrt{\exp(t^{-1}-1)}X^{(1/2)}_{1-\exp(1-t^{-1})}\to0\quad\text{ a.s.}
\end{equation}
From the proof of part (c) of Theorem \ref{link} we know that there exists a Wiener process $(\tilde B_t)_{t\geq0}$ such that $({(1-s)^{-1/2}}X_s^{(1/2)})_{s\in[0,1)}=(\tilde B_{-\log(1-s)})_{s\in[0,1)}$ almost surely. Letting $s=1-\exp(1-t^{-1})$ we get
$$\left(\sqrt{\exp(t^{-1}-1)}X^{(1/2)}_{1-\exp(1-t^{-1})}\right)_{t\in(0,1]}=\left(\tilde B_{t^{-1}-1}\right)_{t\in(0,1]}\quad\text{ a.s.}$$
from which \eqref{cont12} follows by the strong law of large numbers for Brownian motion, since almost surely
$$t\,\sqrt{\exp(t^{-1}-1)}X^{(1/2)}_{1-\exp(1-t^{-1})}=t\,\tilde B_{t^{-1}-1}=(1-t)\,\frac{t}{1-t}\,\tilde B_{\frac{1-t}{t}}\to0$$
as $t\downarrow0$. Hence the centered Gaussian processes under consideration in \eqref{spatim12DMMY} almost surely have continuous sample paths on $[0,1]$ starting in the origin. Thus it remains to show the equality of their covariance functions for $0<s\leq t\leq1$.
Using \eqref{aWCov} and the fact that the function $(0,1]\ni t \mapsto 1-\exp(1-t^{-1})$ is decreasing, we get for any $0<s\leq t\leq1$,
\begin{align*}
\Cov\left(X^{(1/2)}_{1-\exp(1-s^{-1})},X^{(1/2)}_{1-\exp(1-t^{-1})}\right) & =\sqrt{\exp(1-s^{-1})}\sqrt{\exp(1-t^{-1})}(t^{-1}-1)\\
& =\frac{\sqrt{\exp(1-s^{-1})}\sqrt{\exp(1-t^{-1})}}{s\cdot t}\,s(1-t),
\end{align*}
from which \eqref{spatim12DMMY} easily follows. Similarly, for any $0\leq s\leq t$ we get using \eqref{aWCov}
\begin{align*}
\Cov\left(X^{(1/2)}_{1-{\exp(-s)}},X^{(1/2)}_{1-{\exp(-t)}}\right) & =\ee^{-s/2}\ee^{-t/2}\,s
\end{align*}
from which \eqref{spatim12Bou} easily follows.
\end{proof}
\begin{thm}\label{link12}
For any $\beta\in\rr$ we have
\begin{align*}
& \Exp\left[\left(\int_0^1\exp\left(\frac{\beta}{\sqrt{1-s}\,(1-\log(1-s))}\,X^{(1/2)}_s\right)\,\frac{\dd s}{(1-s)\,(1-\log(1-s))^{2}}\right)^{-1}\right]=1\\
\intertext{and}
& \Exp\left[\left(\int_0^{1-\ee^{-1}}\exp\left(\frac{\beta}{\sqrt{1-s}}\,X^{(1/2)}_s\right)\,\frac{\dd s}{1-s}\right)^{-1/2}\right]=1.
\end{align*}
\end{thm}
\begin{proof}
Applying \eqref{spatim12DMMY} to \eqref{DMMY} together with a change of variables $s=1-\ee^{-(t^{-1}-1)}$ yields for any $\beta\in\rr$
\begin{align*}
1 & =\Exp\left[\left(\int_0^1\exp(\beta X^{(1)}_t)\,\dd t\right)^{-1}\right]\\
& =\Exp\left[\left(\int_0^1\exp\left(\beta t\,\sqrt{\exp(t^{-1}-1)}X^{(1/2)}_{1-\exp(1-t^{-1})}\right)\,\dd t\right)^{-1}\right]\\
& =\Exp\left[\left(\int_0^1\exp\left(\frac{\beta}{\sqrt{1-s}\,(1-\log(1-s))}\,X^{(1/2)}_s\right)\,\frac{\dd s}{(1-s)\,(1-\log(1-s))^{2}}\right)^{-1}\right]
\end{align*}
which proves the first identity. Similarly, an application of \eqref{spatim12Bou} to \eqref{Bougerol_special} together with a change of variables $s=1-\ee^{-t}$ yields for any $\beta\in\rr$
\begin{align*}
1 & =\Exp\left[\left(\int_0^1\exp(\beta X^{(0)}_t)\,\dd t\right)^{-1/2}\right]\\
& =\Exp\left[\left(\int_0^1\exp\left(\beta\,\ee^{t/2}X^{(1/2)}_{1-\ee^{-t}}\right)\,\dd t\right)^{-1/2}\right]\\
& =\Exp\left[\left(\int_0^{1-\ee^{-1}}\exp\left(\frac{\beta}{\sqrt{1-s}}\,X^{(1/2)}_s\right)\,\frac{\dd s}{1-s}\right)^{-1/2}\right]
\end{align*}
which proves the second identity.
\end{proof}

\begin{remark}\label{open}
Motivated by the identities \eqref{Bougerol_special} and \eqref{DMMY}, one can formulate the open question
 whether there exists a (continuous) function $p:[0,1]\to(-\infty,0)$ such that
$$\Exp\left[\left( \int_0^1 \exp\left(\beta X^{(\alpha)}_t\right)\,\dd t\right)^{p(\alpha)}\right] = 1 \quad\text{ for every }\beta\in\rr.$$
\end{remark}

\bibliographystyle{plain}

\end{document}